\documentclass[12pt,leqno,twoside]{article}
\usepackage{amssymb}
\usepackage{amsmath}
\usepackage{amsthm,amscd}
\usepackage{a4,indentfirst,latexsym}
\usepackage{graphics}
\usepackage{mathrsfs}
\usepackage{cite,enumitem,graphicx}
\usepackage[colorlinks=true,urlcolor=black,
citecolor=black,linkcolor=black,linktocpage,pdfpagelabels,
bookmarksnumbered,bookmarksopen]{hyperref}
\usepackage[colorinlistoftodos]{todonotes}
\usepackage{color}

\newtheorem{thm}{Theorem}[section]
\newtheorem{prop}[thm]{Proposition}
\newtheorem{lem}[thm]{Lemma}

\newtheorem{rem}[thm]{Remark}

\newenvironment{altproof}[1]
{\noindent
	{\em Proof of {#1}}.}
{\nopagebreak\mbox{}\hfill $\Box$\par\addvspace{0.5cm}}

\newcommand\mytop[2]{\genfrac{}{}{0pt}{}{#1}{#2}}

\newcommand{\eps}{\varepsilon}
\newcommand{\ov}{\overline}

\newcommand\supp{\mathrm{supp}}

\def\N{\mathbb{N}}

\def\R{\mathbb{R}}

\newcommand{\cC}{{\mathcal C}}
\newcommand{\cD}{{\mathcal D}}

\newcommand{\cF}{{\mathcal F}}

\newcommand{\cH}{{\mathcal H}}

\newcommand{\cM}{{\mathcal M}}

\newcommand{\cO}{{\mathcal O}}

\newcommand{\vphi}{\varphi}
\newcommand{\al}{\alpha}

\newcommand{\de}{\delta}
\newcommand{\ka}{\kappa}
\newcommand{\la}{\lambda}

\newcommand{\si}{\sigma}

\newcommand{\De}{\Delta}

\newcommand{\Om}{\Omega}
\newcommand{\Si}{\Sigma}

\newcommand{\pa}{\partial}

\newcommand{\intsigma}{{\mathring\Si}}

\newcommand{\wt}{\widetilde}

\newcommand{\beq[1]}{\begin{equation}\label{eq:#1}}
	\newcommand{\eeq}{\end{equation}}

\numberwithin{equation}{section}

\DeclareMathOperator{\dist}{dist}

\begin{document}
	
	\title{The Morse property of limit functions appearing in mean field equations on surfaces with boundary}
	\author{Zhengni Hu\footnote{Supported by CSC No. 202106010046.} \and Thomas Bartsch\footnote{Supported by DFG grant BA 1009/19-1.}}
	
	\maketitle
	
	
	
	
	
	
	\noindent{\bf Abstract:} 
	In this paper, we study the Morse property for functions related to limit functions of mean field equations on a smooth, compact surface $\Si$ with boundary $\pa \Si$. Given a Riemannian metric $g$ on $\Si$ we consider functions of the form 
	\[
	f_g(x) := \sum_{i=1}^m\si_i^2R^g(x_i)+\sum_{\mytop{i,j=1}{i\ne j}}^m\si_i\si_jG^g(x_i,x_j)+h(x_1,\ldots,x_m),
	\]
	where $\si_i \neq 0$ for $i=1,\ldots,m$,  $G^g$ is the Green function of the Laplace-Beltrami operator on $(\Si,g)$ with Neumann boundary conditions, $R^g$ is the corresponding Robin function, and $h \in \cC^{2}(\Si^m,\R)$ is arbitrary. We prove that for any Riemannian metric $g$, there exists a metric $\wt g$ which is arbitrarily close to $g$ and in the conformal class of $g$ such that $f_{\wt g}$ is a Morse function. Furthermore we show that, if all $\si_i>0$, then the set of Riemannian metrics for which $f_g$ is a Morse function is open and dense in the set of all Riemannian metrics.
	
	
	\section{Introduction and main results}\label{sec:intro}
	Let $\Si$ be a smooth, compact surface with boundary $\pa\Si$. For a Riemannian metric $g$ on $\Si$ let $G^g:\Si\times\Si\to\R$ be the Green function for the Laplace-Beltrami operator $-\De_g$ with Neumann boundary conditions and mean value $0$; i.e.\ for each $\xi\in\Si$ the function $G^g(\cdot,\xi)$ is the unique solution of
	\[
	\left\{
	\begin{aligned}
		-\De_g G^g(\cdot,\xi)&= \de_\xi-\frac1{|\Si|_g} &&\qquad\text{in $\intsigma := \Si\setminus\pa\Si$,}\\
		\pa_{\nu_g} G^g(\cdot,\xi) &= 0 &&\qquad\text{on $\pa\Si$,}\\ 
		\int_\Si G^g(\cdot,\xi)\,dv_g &= 0.
	\end{aligned}
	\right.
	\]
	Here $\nu_g$ is the unit outward normal on $\pa\Si$, $dv_g$ is the area element of $\Si$, $|\Si|_g=\int_{\Si} dv_g$ and  $\de_{\xi}$ is the Dirac distribution on $\Si$ concentrated at $\xi\in\Si$. Setting $\ka(\xi)=2\pi$ for $\xi\in\intsigma$ and $\ka(\xi)=\pi$ for $\xi\in\pa\Si$, the Robin function $R^g:\Si\to\R$ is defined by
	\[
	R^g(\xi) := \lim_{x\to \xi}\left(G^g(x,\xi) + \frac 1 {\ka(\xi)}\log d_g(x,\xi)\right),
	\]
	where $d_g$ denotes the distance with respect to the metric $g$. Given integers $m\ge l\ge 0$ with $m\ge1$ and real numbers $\si_i\ne0$, $i=1,\dots,m$, we set 
	\[
	X := \intsigma^l\times(\pa\Si)^{m-l}\qquad\text{and}\qquad \De_X := \{x\in X:\text{$x_i=x_j$ for some $i\ne j$}\},
	\]
	and define, for an arbitrary given function $h\in \cC^2(\Si^m, \R)$:
	\begin{equation}\label{eq:f_g}
		f_g:X\setminus\De_X \to \R,\quad f_g(x) = \sum_{i=1}^m\si_i^2R^g(x_i)+\sum_{\mytop{i,j=1}{i\ne j}}^m\si_i\si_jG^g(x_i,x_j)+h(x_1,\cdots,x_m).
	\end{equation}
	We will prove that $f_g$ is a Morse function for a ``generic'' metric $g$ on $\Si$, in a sense that will be made precise below.
	
	Functions of the form \eqref{eq:f_g} play a crucial role in understanding the blow-up behavior of solutions of mean field equations like
	\begin{equation}\label{eq:mf}
		\begin{cases}
			-\De_g u= \lambda\left( \frac{V e^u}{\int_{\Si} Ve^u\, dv_g} - \frac 1 {|\Si|_g}\right) &  \text{ in } \intsigma, \\
			\pa_{\nu_g} u = 0 & \text{ on } \pa \Si, \\
			\int_{\Si} u\, dv_g = 0. &
		\end{cases}
	\end{equation}
	For a compact surface without boundary it has been proven in \cite{Esposito2014singular} that a nondegenerate or more generally, an isolated stable critical point $x=(x_1,\dots,x_m)$ of $f_g$, with $l=m$ and $h(x) = 2 \sum_{i=1}^m\si_i\log V(x_i)$, gives rise to solutions $(\la,u)$ of \eqref{eq:mf} with $\la$ close to $8\pi m$ and such that $u$ blows up as $\la\to8\pi m$ precisely at $x_1,\dots,x_m\in\Si$. Similar results have been obtained in the case of mean field equations on a bounded domain in $\R^2$ with Dirichlet boundary conditions in \cite{ma_convergence_2001, Esposito2005}, and for solutions of Gelfand's problem and steady states of the Keller-Segel system (see \cite{baraket1998construction, gladiali2004some, del_pino_singular_2005, del_pino_collapsing_2006} for instance).
	
	Now we state our main results. We fix an integer $k\ge0$ and  $0<\al<1$, and let $\mathrm{Riem}^{k+2,\al}(\Si)$ be the space of Riemannian metrics of class $\cC^{k+2,\al}$ on $\Si$, i.e.\ the space of symmetric and positive definite sections $\Si\to(T\Si\otimes T\Si)^*$ of class $\cC^{k+2,\al}$. 
	
	\begin{thm}\label{thm:main1}
		If $\si_i\ne0$ for $i=1,\dots,m$ and $\si_1+\dots+\si_m\ne0$ then for any $g\in \mathrm{Riem}^{k+2,\al}(\Si)$ the set 
		\[
		\cM_g^{k+2,\al}(\Si) := \{\psi\in\cC^{k+2,\al}(\Si,\R_+):\text{$f_{\psi g}$ is a Morse function} \}
		\]
		is a residual subset of $\cC^{k+2,\al}(\Si,\R_+)$.
	\end{thm}
	
	As a consequence of Theorem~\ref{thm:main1} for any Riemannian metric $g$ on $\Si$, there exists a metric $\wt g$ which is arbitrarily close to $g$ and conformal to $g$, and such that $f_{\wt g}$ is a Morse function. 
	
	\begin{thm}\label{thm:main2}
		If $\si_i>0$ for all $i=1,\dots,m$ then the set
		\[
		\mathrm{Riem}_{Morse}^{k+2,\al}(\Si) := \{g\in \mathrm{Riem}^{k+2,\al}(\Si):\text{$f_g$ is a Morse function} \}
		\]
		is an open and dense subset of $\mathrm{Riem}^{k+2,\al}(\Si)$.
	\end{thm}
	
	\begin{rem}
		In \cite{Ahmedou-Bartsch-Fiernkranz:2023} the authors considered functions $f_g$ of the form \eqref{eq:f_g} on a surface $\Si$ without boundary and with $h(x_1,\dots,x_m)=\sum_{i=1}^m\log V(x_i)$ where $V\in\cC^2(\Si,\R_+)$. This is motivated by the mean field equation \eqref{eq:mf}. They proved that $f_g$ is a Morse function for $V$ in an open and dense subset of $\cC^2(\Si,\R_+)$. Also related is the paper \cite{Bartsch2017TheMP}, which deals with functions like $f_g$ where $G$ is the Green function of the Dirichlet Laplace operator on a bounded smooth domain $\Om\subset\R^n$. In \cite{Bartsch2017TheMP}, it is proved that $f_g$ is a Morse function for a generic domain $\Om$. The present paper seems to be the first investigating the Morse property of $f_g$ as in \eqref{eq:f_g} as a function of the Riemannian metric.
	\end{rem}
	
	\section{Preliminaries}\label{sec:prelim}
	
	Riemann surfaces are locally conformally flat, so there exist isothermal coordinates where the metric is conformal to the Euclidean metric (see \cite{chern1955, Hartman1955, Vekua1955}). We modify the isothermal coordinates applied in \cite{Esposito2014singular, figueroa2022bubbling,yang_zhou2021} to a Riemann surface with boundary. For $\xi\in\Si$ there exists a local chart $y_{\xi}:\Si\supset U(\xi)\to B^\xi\subset\R^2$ transforming $g$ to $e^{\vphi_\xi\circ y_\xi}\langle\cdot,\cdot\rangle_{\R^2}$. We may assume that $y_\xi(\xi)=0$. For $\xi\in\intsigma$ we may also assume that $\ov{U(\xi)}\subset \intsigma$ and that the image of $y_{\xi}$ is a disc $B^\xi:=\{ y\in \R^2: |y|< 2r_{\xi}\}$ of radius $2r_\xi>0$. For $\xi\in\pa\Si$, by Lemma 4 in \cite{yang_zhou2021},  there exists an isothermal coordinate system $\left(U(\xi), y_{\xi}\right)$ near $\xi$ such that the image of $y_{\xi}$ is a half disk ${B}_{2r_{\xi}}^{+}:=\{y\in\R^2: |y|<2r_{\xi}, y_2\geq 0\}$ of radius $2r_{\xi}>0$ and $y_{\xi}\left(U(\xi)\cap \partial \Sigma\right)= {{B}_{2r_{\xi}}^{+}} \cap \partial \mathbb{R}^{2}_+$.
	For any $x \in$ $y_{\xi}^{-1}\left({{B}_{2r_{\xi}}^{+}} \cap \partial \mathbb{R}^{2}_+\right)$, we have
	\begin{equation}\label{eq:out_normal_derivatives}
		\left(y_{\xi}\right)_*(\nu_g(x))=\left. -\exp\left( -\frac{\varphi_{\xi}(y)}2\right) \frac {\partial} { \partial y_2 }\right|_{	y=y_{\xi}(x)}.
	\end{equation}
	In this case, we take $B^\xi = {B}_{2r_{\xi}}^{+}$. 
	For $\xi\in \Si$ and $0<r\le 2r_\xi$ we set
	\[
	B_r^\xi := B^\xi \cap \{ y\in\R^2: |y|< r\}\quad \text{and}\quad U_{r}(\xi):=y_\xi^{-1}(B_{r}^{\xi}).
	\]
	Recall that $\vphi_\xi:B^\xi\to\R$ is related to $K$, the Gaussian curvature of $\Si$, by the equation
	\[
	-\De \vphi_\xi(y) = 2K\big(y^{-1}_\xi(y)\big) e^{\vphi_\xi(y)} \quad\text{for all $y\in B^\xi$.} 
	\]
	We can assume that $y_\xi$ and $\vphi_\xi$ depend smoothly on $\xi$, and that $\vphi_\xi(0)=0$ and $\nabla\vphi_\xi(0)=0$. 
	
	Observe that we have for $\zeta\in U(\xi)$, 
	\[
	\lim_{x\to\zeta}\frac{d_g(x,\zeta)}{\big|\vphi_\xi(x)-\vphi_\xi(\zeta)\big|} = e^{\frac12\vphi_\xi(\zeta)},
	\]
	which implies\begin{equation}\label{eq:robin}
		R^g(\zeta) = \lim_{x\to\zeta}\left(G^g(x,\zeta)+\frac1{\ka(\zeta)}\log\big|y_\xi(x)-y_\xi(\zeta)\big|\right)
		+ \frac1{2\ka(\zeta)}\vphi_\xi\big(y_\xi(\zeta)\big),
	\end{equation}
	and in particular, using $\vphi_\xi\big(y_\xi(\xi)\big)=\vphi_\xi(0)=0$:
	\[
	R^g(\xi) = \lim_{x\to\xi}\left(G^g(x,\xi)+\frac1{\ka(\xi)}\log\big|y_\xi(x)\big|\right) . 
	\]
	
	Now we construct a regular part $H^g(x,\xi)$ of the Green function $G^g(x,\xi)$ such that $H^g(\xi,\xi)=R^g(\xi)$. Let $\chi\in\cC^\infty(\R,[0,1])$ be such that 
	\[
	\chi(s)
	= \begin{cases}
		1 &\quad \text{if $|s|\leq 1$,}\\
		0 &\quad \text{if $|s|\geq 2$.}
	\end{cases}
	\]
	For $\xi\in\intsigma$ we choose $\de_\xi = \min\{\frac 1 2 r_\xi,\frac12\dist(x,\pa\Si)\}$, and for $\xi\in\pa\Si$ we set $\de_\xi:=\frac 1 2 r_\xi$. Next we define the cut-off function $\chi_\xi\in\cC^\infty(\Si,[0,1])$ by 
	\[
	\chi_\xi(x) := \begin{cases}
		\chi\big(|y_\xi(x)|/\de_\xi\big) &\quad \text{if $x\in U(\xi)$}\\
		0&\quad \text{if $x\in \Si\setminus U(\xi)$}
	\end{cases}. 
	\]
	Then, for $\xi\in\Si$ the function $H^g_{\xi}:=H^g(\cdot,\xi):\Si\to\R$ is defined to be the unique solution of the Neumann problem
	\begin{equation}\label{eqR}
		\left\{
		\begin{aligned}
			\De_g H^g_{\xi}
			&= \frac1{\ka(\xi)} (\De_g\chi_\xi)\log\big|y_\xi\big|+\frac2{\ka(\xi)} \left\langle\nabla^g\chi_\xi,\nabla^g\log|y_\xi|\right\rangle_g+\frac1{|\Si|_g}
			&&\quad \text{in $\intsigma$,}\\
			\pa_{ \nu_g} H^g_{\xi}
			&= \frac1{\ka(\xi)}\big(\pa_{\nu_g }\chi_\xi\big) \log|y_\xi| + \frac1{\ka(\xi)}\chi_\xi \pa_{\nu_g}  \log|y_\xi|	
			&&\quad \text{on $\pa\Si$,}\\
			\int_{\Si} H^g_{\xi} dv_g &= \frac1{\ka(\xi)}\int_{\Si} \chi_{\xi}\log|y_\xi|\, dv_g .
		\end{aligned}
		\right.
	\end{equation}
	
	\begin{lem}\label{lem:green}
		For $g\in \mathrm{Riem}^{k+2,\al}(\Si)$ the function $H^g$ is of class $\cC^{k+3,\al}$ in any compact subsets of $\Si\times\intsigma$ and in $\Si\times\pa\Si$. Moreover, it satisfies 
		\[
		G^g(x,\xi) = -\frac1{\ka(\xi)}\chi_\xi(x)\cdot\log\big|y_\xi(x)\big| + H^g(x,\xi)
		\]
		and $H^g(\xi,\xi)=R^g(\xi)$. Consequently $R^g$ is of class $\cC^{k+3,\al}$ in any compact subsets of $\intsigma$ and in $\pa\Si$.
	\end{lem}
	
	\begin{proof}
		First we observe that $e^{\varphi_{\xi}}\in \cC^{k+2,\al}(B^{\xi}, \R)$ because $g\in \mathrm{Riem}^{k+2,\al}(\Si) $. For $\xi\in \intsigma$, by the choice of $\delta_{\xi}$ we have that $-\De_g H^g_\xi$ is of class $\cC^{k+2,\al}$ class in $\Si$ and $\pa_{\nu_g } H^g(x,\xi) \equiv 0$ on $\pa \Si$. Now the Schauder estimate for the Neumann problem (see \cite{Nardi2014,yang_zhou2021}, for instance) implies that the solution of \eqref{eqR} uniquely exists in $\cC^{k+4,\al}(\Si,\R)$. 
		
		For $x\in U(\xi)\cap \pa \Si$, setting $y=y_{\xi}(x)$ we have:
		\begin{equation*}
			\pa_{ \nu_g } \log|y_{\xi}(x)|
			\stackrel{\eqref{eq:out_normal_derivatives}}{=} - e^{-\frac 1 2 {\varphi}_{\xi}(y)}\frac {\pa}{\pa y_2} \log |y|
			= - e^{-\frac 1 2 {\varphi}_{\xi}(y)}\frac{y_2}{|y|^2}\equiv 0. 
		\end{equation*}
		Clearly, $\pa_{\nu_g} \chi(|y_{\xi}|(x))=0$ for $x \in \pa\Si \cap U_{\de_{\xi}}(\xi)$. It follows that $\pa_{\nu_g} H^g_\xi$ is of class $\cC^{k+2,\al}$ on $\pa \Si$. Moreover $\De_{g}H^g_\xi$ is of class $\cC^{k+2,\al}$ in $\Si$. Consequently \eqref{eqR} has a unique solution $H^g_\xi \in \cC^{k+3,\al}(\Si,\R)$ by the Schauder estimates. 
		
		Finally $H^g_\xi$ is uniformly bounded in $\cC^{k+3, \al}$ for $\xi$ in any compact subsets of $\intsigma$ and in $\pa\Si$. Therefore $ H^g(\xi,\xi)$ is in $\cC^{k+3, \al}$ in any compact subsets of $\intsigma$ and in $\pa\Si$, and $H^g(\xi,\xi)=R^g(\xi)$ by \eqref{eq:robin}.
	\end{proof}
	
	For $g\in\mathrm{Riem}^{k+2,\al}(\Si)$ we now consider the map
	\[
	\cH^g: \Si\times\Si\times\cC^{k+2,\al}(\Si,\R_+) \to \R,\quad \cH^g(x,\xi,\psi) := H^{\psi g}(x,\xi).
	\]
	
	\begin{prop}\label{prop:dcH}
		The map $\cH^g$ is $\cC^1$ in $\Si\times\intsigma\times\cC^{k+2,\al}(\Si,\R_+)$ and in $\Si\times\pa\Si\times\cC^{k+2,\al}(\Si,\R_+)$. Moreover,  we have 
		\begin{equation}~\label{expansion_D1}
			D_{\psi} \cH^g(x,\xi, 1)[\theta] = - \frac {1}{|\Si|_g}\int_{\Si} (G^{g}(z,x) + G^g(z,\xi)) \theta(z) dv_g(z),
		\end{equation}
		for any   $\theta\in\cC^{k+2,\al}(\Si,\R)$.
	\end{prop}
	
	\begin{proof}
		For $g\in \mathrm{Riem}^{k+2,\al}(\Si)$ and $\psi\in \cC^{k+2,\al}(\Si,\R_+)$ we clearly have $\psi g\in \mathrm{Riem}^{k+2,\al}(\Si)$. By a direct calculation we obtain the following equations for $H^{\psi g}_\xi - H^g_\xi$: 
		\begin{equation*}
			\left\{
			\begin{aligned}
				-\De_{g} (H^{\psi g}_\xi - H^g_\xi) &= \frac 1 {|\Si|_g} - \frac {\psi}{ |\Si|_{\psi g}} && \text{in $\intsigma$,}\\
				\pa_{ \nu_g} (H^{\psi g}_\xi - H^g_\xi) &= 0 && \text{on $\pa\Si$,}\\
				\int_{\Si} (H^{\psi g}_\xi -H^{{g}}_\xi)\, dv_g &= -\int_{\Si} G^{\psi g}_\xi (\psi-1) dv_g.
			\end{aligned} 
			\right.
		\end{equation*}
		An expansion of $|\Si|_{\psi g}= \int_{\Si} dv_{\psi g}$ yields:
		$$
		-\De_g(H^{\psi g}_\xi - H^g_\xi)
		= \frac{1-\psi}{|\Si|_g}+ \frac{\int_{\Si}(\psi-1) dv_g}{|\Si|^2_{g}}+ \mathcal{ O}(\|\psi-1\|_{\cC^{k+2,\al}}^2)
		\quad\text{as $\|\psi-1\|_{\cC^{k+2,\al}} \to 0$.} 
		$$
		Recall that $G^{\psi g}_\xi =  - \frac 4 {\si(\xi)}\chi(|y_{\xi}|)\log {|y_{\xi}|}+H^{\psi g}_\xi$, so the representation formula gives: 
		\begin{equation}~\label{eq:diffe_H}
			H^{\psi g}_\xi(x) - H^g_\xi(x)
			= \frac 1 {|\Si|_g} \int_{\Si} G^{\psi g}(z,\xi)(\psi(z)-1)\,dv_g(z) -\frac 1 {|\Si|_{\psi g}} \int_{\Si} G^g(z,x)(\psi(z)-1)\,dv_g(z). 
		\end{equation}
		By standard elliptic estimates (see~\cite{Nardi2014,yang_zhou2021}) there exists a constant $C$ such that 
		\[
		\big\|H^{\psi g}_\xi- H^g_\xi\big\|_{\cC^{k+4,\al}} \leq C\cdot\|\psi-1\|_{\cC^{k+2,\al}},
		\]
		thus
		\begin{equation}\label{eq:H_C0}
			H^{\psi g}_\xi \to  H^g_\xi \quad\text{ in $\cC^{k+4,\al}$ as $\psi \to 1 $ in $\cC^{k+2,\al}$.}
		\end{equation}
		According to the construction of $H^g(x,\xi)$, the convergence in \eqref{eq:H_C0} is uniform for $\xi$ in any compact subset of $\intsigma$, and in $\pa \Si$. 
		It follows that $\cH_g(x,\xi,\cdot)$ is continuous at $1$, uniformly for $x\in \Si$ and $\xi$ in any compact subsets of $\intsigma$ or $\xi\in\pa\Si$. Using
		$ \cH_{g}(x,\xi,\psi)= \cH_{\psi g} (x,\xi,1)$ we see that $\cH_g(x,\xi, \cdot)$ is continuous at every $\psi\in \cC^{k+2,\al}(\Si,\R_+)$. 
		
		Next we prove that $\cH^g(x,\xi,\psi)$ is $\cC^1$ with respect to $\psi$. We fix $\theta\in \cC^{2+k,\al}(\Si, \R)$ and consider the metric $(1+t\theta)g$ with $t$ sufficiently small so that $1+t\theta > 0$. Then 
		\begin{equation}\label{eq:w_t}
			w_\xi^t(x) :=\frac1t\big(H^{(1+t\theta)g}(x,\xi)-H^g(x,\xi)\big) = \frac1t\big(\cH^g(x,\xi,1+t\theta)-\cH^g(x,\xi,1)\big)
		\end{equation}
		satisfies the following equations as $t\to0$:
		\begin{equation}~\label{eq:RH_C1}
			\left\{ \begin{aligned}
				-\De_{g}w_\xi^t &= \frac 1 {|\Si|^2_g}\int_{\Si} \theta\, dv_{g} -\frac{\theta}{|\Si|_g} + \cO(t) &&\quad \text{in $\intsigma$},\\
				\pa_{\nu_g} w_t(x,\xi) &= 0 &&\quad x\in \pa\Si,\\
				\int_{\Si} w_\xi^t\, dv_{g} &=- \int_{\Si} G^{(1+t\theta)g}_\xi \cdot \theta\, dv_g 
			\end{aligned} \right.,
		\end{equation}
		where $\cO(t)$ is defined with respect to the $\cC^{k+2,\al}$-norm. Applying the standard elliptic estimates, $w_\xi^t$ converges as $t\to0$ to some function $w_\xi^0(x) = D_\psi\cH^g(x,\xi,1)[\theta]$ in $\cC^{k+4,\al}(\Si, \R)$. Moreover, $w_\xi^0$ satisfies the equations
		\begin{equation}~\label{eq:H_C1}
			\left\{
			\begin{aligned}
				-\De_{g}w_\xi^0 &= \frac 1 {|\Si|^2_g}\int_{\Si} \theta\, dv_{g} -\frac{\theta}{|\Si|_g} &&\quad  \text{ in $\intsigma$,} \\
				\pa_{\nu_g} w_\xi^0 &= 0 &&\quad \text{ on $\pa\Si$,}\\
				\int_{\Si} w_\xi^0\, dv_{g} &= -\int_{\Si} G^{g}_\xi \cdot \theta\, dv_g.
			\end{aligned} 
			\right.
		\end{equation}
		The representation formula now implies \eqref{expansion_D1}:
		\[
		D_\psi\cH^g(x,\xi,1)[\theta] = w_\xi^0(x) = -\frac 1 {|\Si|_{g}}\int_{\Si} \big(G^{g}(z,x) +G^{g}(z,\xi)\big) \cdot \theta(z)\, dv_g(z).  
		\]
		Replacing $g$ by $\psi g$ and $\theta$ by $\frac\theta\psi$, we obtain the derivative of $\cH_g(x,\xi,\psi)$ for arbitrary $\psi\in \cC^{k+2,\al}(\Si,\R_+)$:
		\begin{equation}\label{eq:HH_C1}
			\begin{aligned}
				D_{\psi}\cH^g(x,\xi,\psi)[\theta] &= \lim_{t\to0}\frac1t\big(H^{(\psi+t\theta) g}(x,\xi)-H^{\psi g}(x,\xi)\big) \\
				&= -\frac{1}{|\Si|_{\psi g}}\int_{\Si} \left( G^{\psi g} (z,x)+G^{\psi g}(z,\xi) \right) \cdot \frac{\theta(z) }{\psi(z)}dv_{\psi g}(z)\\
				&= -\frac{1}{|\Si|_{\psi g}}\int_{\Si} \left( G^{\psi g} (z,x)+G^{\psi g}(z,\xi) \right) \cdot \theta(z) dv_{g}(z).
			\end{aligned}
		\end{equation}
		In order to see that $\cH_g(x,\xi,\psi)$ is continuously Fr\'echet differentiable with respect to $\psi$ it is sufficient to prove that $D_{\psi}\cH_g(x,\xi,\psi)$ is continuous in $\psi$ as a linear operator on $\cC^{k+2,\al}(\Si,\R_+)$. For $\psi_1,\psi_2\in \cC^{k+2,\al}(\Si,\R_+)$ there holds
		\begin{equation*}
			\begin{aligned}
				\big| D_{\psi}\cH_g(x,\xi,\psi_1)[\theta]-D_{\psi}\cH_g(x,\xi,\psi_2)[\theta] \big|
				&= \left| \frac{1}{|\Si|_{\psi_1 g}}\int_{\Si} \left( G^{\psi_1 g} (z,x)+G^{\psi_1 g}(z,\xi) \right)\theta(z) dv_{g}(z)\right.\\
				&\hspace{1cm} \left. - \frac{1}{|\Si|_{\psi_2 g}}\int_{\Si} \left( G^{\psi_2 g} (z,x)+G^{\psi_2 g}(z,\xi) \right)\theta(z) dv_{g}(z)\right|  \\
				&\le C\cdot\|\theta\|_{\cC^{k+2,\al}} \cdot \|\psi_1-\psi_2\|_{\cC^{k+2,\al}}
			\end{aligned}
		\end{equation*}
		where we applied \eqref{eq:HH_C1}; here $C>0$ is a constant. Therefore $\cH^g$ is $\cC^1$ in $\Si\times \intsigma \times \cC^{k+2,\al}(\Si,\R_+)$ and in $\Si\times\pa\Si\times \cC^{k+2,\al}(\Si,\R_+).$
	\end{proof}
	
	\section{Proof of Theorem~\ref{thm:main1}}\label{sec:proof1}
	The proof is based on the following transversality theorem \cite[Theorem~5.4]{henry2005perturbation}.
	
	\begin{thm}~\label{thm:trans}
		Let $M,\Psi,N$ be Banach manifolds of class $\cC^r$ for some $r\in \N$, let $\cD\subset M\times \Psi$ be open, let $\cF: \cD\to N$ be a $\cC^r$ map, and fix a point $z\in N$. Assume for each $(y,\psi)\in \cF^{-1}(z)$ that: 
		\begin{itemize}
			\item[(1)] $D_y\cF(y,\psi): T_y M\to T_{z} N$ is semi-Fredholm with index $<r$;
			\item[(2)] $D\cF(y,\psi) : T_y M \times T_\psi \Psi \to T_{z} N$ is surjective;
			\item[(3)]  $\cF^{-1}(z)\to \Psi$, $(y,\psi)\mapsto \psi$, is $\sigma$-proper. 
		\end{itemize}
		Then 
		\[
		\Psi_{reg}:=\{ \psi\in \Psi: z \text{ is a regular value of } \cF(\cdot, \psi)\}
		\]
		is a residual subset of $\Psi$. 
	\end{thm}
	
	\begin{altproof}{Theorem~\ref{thm:main1}}
We set $\Psi := \cC^{k+2,\al}(\Si,\R_+)$ and consider the functions $f_{\psi g}:X \setminus \De_X \to \R$ for $\psi\in\cC^{k+2,\al}(\Si,\R_+)$ first in a local isothermal chart. Let $y_{\xi_i}:\Si\supset U(\xi_i) \to B^{\xi_i} \subset \R^2$ be isothermal charts of $\Si$ as in Section~\ref{sec:prelim} with $\xi_1,\dots,\xi_l\in\intsigma$, $\xi_{l+1},\dots,\xi_m\in\pa\Si$. For simplicity of notation we set $Y_i:=y_{\xi_i}:U_i:=U(\xi_i) \to B_i:=B^{\xi_{i}}\subset\R^2$ for $i=1,\dots,l$, and $Y_i:=\pi_1\circ y_{\xi_i}:U_i:=U(\xi_i)\cap\pa\Si \to B_i \subset \R$ for $i=l+1,\dots,m$; here $\pi_1:\R^2\to\R$ is the projection onto the first component. For $i=l+1,\dots,m$ we thus have $y_{\xi_i}(x) = (Y_i(x),0)$. Then
	\[
		Y=Y_1\times\dots\times Y_m: X \supset U:=U_1\times\cdots\times U_m \to \R^{l+m},\ \ 
		(x_1,\dots,x_m) \mapsto \big(Y_1(x_1),\dots,Y_m(x_m)\big),
	\]
is a chart of $X$.  Set $M=N:=\R^{l+m}$ and $V := Y(U\cap X\setminus \De_X) \subset \R^{l+m}=M$ so that $\cD := V \times \cC^{k+2,\al}(\Si,\R_+)$ is an open subset of $\R^{l+m} \times \cC^{k+2,\al}(\Si,\R_+)$. 
		
		We will apply Theorem~\ref{thm:trans} to prove that $0\in \R^{l+m}$ is a regular value of 
		\[
		\nabla \big(f_{\psi g}\circ Y^{-1}\big): \R^{l+m}\supset V=Y(U\cap X\setminus \De_X) \to \R^{l+m}
		\]
		for $\psi$ in a residual subset $\Psi_U \subset \Psi$. This implies that the restriction of $f_{\psi g}$ to $U\cap X\setminus \De_X$ is a Morse function for $\psi\in\Psi_U$. Then Theorem~\ref{thm:main1} follows because $X$ is covered by finitely many neighborhoods $U$ as above and because the intersection of finitely many residual sets is a residual set. 
		
It remains to prove that the map
	\[
		\cF_g: \R^{l+m}\times\cC^{k+2,\al}(\Si,\R_+) \supset \cD\to \R^{l+m},\quad
		(y,\psi) mapsto \nabla\big(f_{\psi g}\circ Y^{-1}\big)(y),
	\]
satisfies the assumptions of Theorem~\ref{thm:trans} with $r=1$. Concerning the differentiability it is clear that $\cF_g$ is $\cC^1$ as a function of $y\in V$. In order to see that $\cF_g$ is also $\cC^1$ in $\psi$, by Lemma~\ref{lem:green} it is sufficient to prove that $\nabla_y \left(H^{\psi g}\left(Y_i^{-1}(\cdot),Y_j^{-1}(\cdot)\right)\right)$ is $\cC^1$ in $\psi$. We recall that $w^t_{\xi}$ is defined by~\eqref{eq:w_t}. Applying the representation formula of $w^t_{\xi}$ and Lebesgue's dominated convergence theorem, we have for $(y,\psi)\in\cD$ and $\theta \in \cC^{k+2,\al}(\Si,\R)$:
		\[
		\begin{aligned}
			& D_{\psi}\big|_{\psi=1}\nabla_y H^{\psi g}(Y_i^{-1}(y_i), Y_j^{-1}(y_j) )[\theta]
			=\lim_{t\rightarrow 0}\nabla_y \left(\left. w^t_{\xi}(x)\right|_{x= Y^{-1}_i(y_i),\xi=Y^{-1}_{j}(y_j)}\right)\\
			&\hspace{1cm}\stackrel{\eqref{eq:diffe_H}}{=} \lim_{t\to 0}\left( -\frac 1 {|\Si|_g}\int_{\Si} \nabla_y G_{Y_{j}^{-1}(y_j)}^{(1+t\theta)g}\theta dv_g 
			+\int_{\Si} \frac 1 t \left( \frac 1 {|\Si|_g}-\frac{1+t\theta}{|\Si|_{(1+t\theta)g}}\right)\nabla_y G^g_{Y^{-1}_i(y_i)}dv_g \right)\\
			&\hspace{1cm}= - \frac {1}{|\Si|_g}\int_{\Si} \nabla_y  \big(G^{g}(z,Y_i^{-1}(y_i) ) + G^g(z,Y_j^{-1}(y_j))\big)\cdot \theta(z) \,dv_g(z)\\
			&\hspace{1cm}=\nabla_y D_{\psi} \cH^g(Y_i^{-1}(y_i),  Y_j^{-1}(y_j), 1) [\theta],
		\end{aligned}
		\]
		where we used Proposition~\ref{prop:dcH}. Since $ D_{\psi}\cF_g(y,\psi)[\theta] =  D_{\psi}\cF_{\psi g}(y, 1)\left[\frac \theta \psi\right]$ we obtain 
		\[
		D_{\psi}\cF_g(y,\psi)[\theta]
		= - \frac {1}{|\Si|_{\psi g}}\int_{\Si} \nabla_y \big(G^{\psi g }(z,Y_i^{-1}(y_i) ) + G^{\psi g}(z,Y_j^{-1}(y_j))\big)\cdot \theta(z) \,dv_{g}(z).  
		\]
		As in the proof of Proposition~\ref{prop:dcH} we deduce that $ \cF_g(y,\psi)$ is $\cC^1$ on $U$.
		
		Now we need to check the assumptions (1)-(3) of Theorem~\ref{thm:trans}. Obviously $D_y\cF_g(y,\psi):\R^{l+m}\to\R^{l+m}$ is a Fredholm operator of index $0<1$, hence (1) holds. Also, (3) is easy to prove: For $j\in\N$ the set 
		\[
		M_j := \big\{Y(x)\in\R^{l+m}: x\in U,\ d_g\big(x,\De_X \cup \pa U\big) \ge 2^{-j}\big\} \subset Y(U) \subset \R^{l+m}
		\]
		is compact as a continuous image of a compact set. Therefore the map
		\[
		M_j\times \cC^{k+2,\al}(\Si,\R_+) \to \cC^{k+2,\al}(\Si,\R_+),\quad (y,\psi) \mapsto \psi,
		\]
		is proper, hence its restriction to $\cF_g^{-1}(0) \cap \big( M_j\times \cC^{k+2,\al}(\Si,\R_+)\big)$ is proper. Since $V = \bigcup_{j=1}^\infty M_j$, so $\cD=\bigcup_{j=1}^\infty \big(M_j \times \cC^{k+2,\al}(\Si,\R_+)\big)$ it follows that the map
		\[
		\cF_g^{-1}(0) = \bigcup_{j=1}^\infty \Big(\cF_g^{-1}(0) \cap \big( M_j\times \cC^{k+2,\al}(\Si,\R_+)\big)\Big) \to \cC^{k+2,\al}(\Si,\R_+),
		\quad (y,\psi) \mapsto \psi,
		\]
		is $\si$-proper.
		
		Finally we prove the surjectivity of the derivative $D\cF_g(y,\psi): \R^{l+m}\times\cC^{k+2,\al}(\Si,\R)  \to \R^{l+m}$ at a point $(y,\psi)\in \cF_g^{-1}(0)$. In fact, we shall prove that $D_\psi\cF_g(y,\psi):\cC^{k+2,\al}(\Si,\R) \to \R^{l+m}$ is onto. Since
		\[
		D_{\psi} \cF_g(y, \psi)[\theta] = D_{\psi} \cF_{\psi g}(y, 1)\left[\theta/ \psi\right]\qquad\text{for $\theta\in\cC^{2+k,\al}(\Si, \R)$}
		\]
		it is sufficient to consider the case $\psi\equiv1$. We observe for $\theta\in\cC^{2+k,\al}(\Si, \R)$:
		\[
		D_{\psi}\big|_{\psi=1}\nabla_y G^{\psi g}(Y_i^{-1}(y_i), Y_j^{-1}(y_j) )[\theta]
		= D_{\psi}\big|_{\psi=1}\nabla_y H^{\psi g}(Y_i^{-1}(y_i), Y_j^{-1}(y_j) )[\theta].
		\] 
		Now Proposition~\ref{prop:dcH} yields for $\theta\in\cC^{2+k,\al}(\Si, \R)$ with $\supp(\theta)\subset\Si\setminus\big\{Y_1^{-1}(y_1),\dots,Y_m^{-1}(y_m)\big\}$:
		\[
		\begin{aligned}
			D_{\psi} \cF_g(y, 1)[\theta]
			&= \frac{d}{dt}\bigg|_{t=0}\Bigg( \sum_{i=1}^m \si_i^2 \nabla_y R^{(1+t\theta)g}(Y_i^{-1}(y_i) ) \\
			&\hspace{2cm}  + \sum_{\mytop{i,j=1}{i\ne j}}^m\si_i\si_j\nabla_y G^{(1+t\theta)g}(Y_i^{-1}(y_i), Y_j^{-1}(y_j)) 
			+ \nabla_y\big( h\circ Y^{-1}\big)(y)\Bigg)\\
			&= -\frac 1 {|\Si|_g} \int_{\Si} \Bigg(\sum_{i=1}^m2\si^2_i \nabla_y G^{g}(z, Y_i^{-1}(y_i) ) \\
			&\hspace{2.7cm}  + \sum_{\mytop{i,j=1}{i\ne j}}^m \si_i\si_j\big(\nabla_y G^g(z, Y_i^{-1}(y_i)) 
			+ \nabla_y G^g(z, Y_j^{-1}(y_j))\big) \Bigg) \theta(z)\, dv_g(z). 
		\end{aligned} 
		\]
		Consider an element $u=(u_1,\cdots,u_m)\in\R^{l+m}$ with $u_1,\dots,u_l\in\R^2$, $u_{l+1},\dots,u_m\in\R$, that is orthogonal to the range of $D_\psi\cF_g(y,1)$, i.e.\ it satisfies for every $\theta \in \cC^{2+k,\al}(\Si, \R)$ with $\supp(\theta) \subset \Si \setminus \big\{Y_1^{-1}(y_1),\dots,Y_m^{-1}(y_m)\big\}$:
		\[
		\begin{aligned}
			0 &= \big\langle u,D_\psi\cF_g(y,1)[\theta]\big\rangle = \sum_{i=1}^m \big\langle u_i, (D_\psi\cF_g(y,1)[\theta])_i\big\rangle\\
			&= - \sum_{i=1}^m \frac1{{|\Si|_g} } \bigg(2\si_i^2 + 2\sum_{\mytop{j=1}{j\ne i}}^m \si_i\si_j \bigg)
			\int_{\Si} \big\langle u_i,\nabla_{y_i} G^g(z, Y_i^{-1}(y_i))\big\rangle\cdot \theta(z) \,dv_g(z)
		\end{aligned}
		\]
		This implies, using $\sum_{j=1}^m\si_j\ne0$:
		\begin{equation}\label{eq:u-ortho}
			\sum_{i=1}^m \si_i\big\langle u_i,\nabla_{y_i} G^g(z, Y_i^{-1}(y_i))\big\rangle  = 0
			\qquad\text{for $z\in \Si\setminus \big\{Y_1^{-1}(y_1),\dots,Y_m^{-1}(y_m)\big\}$.} 
		\end{equation}
		Setting $\ka_i=2\pi$ for $i=1,\dots,l$ and $\ka_i=\pi$ for $i=l+1,\dots,m$ we have
		\[
		G^g(z,Y_i^{-1}(y_i)) = H^g\big(z,Y_i^{-1}(y_i)\big) - \frac 1{\ka_i} \chi\big(4|Y_i(z)-y_i|/\delta_{\xi_i}\big)\log|Y_i(z)-y_i|.
		\]
		Now we define $z_i(t) := Y_i^{-1}(y_i+tu_i)$ for $i\in\{1,\dots,m\}$ and observe that
		\[
		\nabla_{y_j}G^g\big(z_i(t),Y_j^{-1}(y_j)\big) = \cO(1)\qquad\text{as $t\to0$ for $j\ne i$}
		\]
		whereas 
		\[
		\big\langle u_i,\nabla_{y_i} G^g(z_i(t), Y_i^{-1}(y_i))\big\rangle  \to \infty\qquad\text{as $t\to0$.}
		\]
		Equation \eqref{eq:u-ortho} implies $u_i=0$ because $\si_i \ne 0$. This holds for all $i$, hence $u=0$ and $D_\psi\cF_g(y,\psi)$ must be onto.
	\end{altproof}

	\section{Proof of Theorem~\ref{thm:main2}}\label{sec:proof2}
	Theorem~\ref{thm:main2} follows easily from the following lemma.
	
	\begin{lem}~\label{lem:H_INFTY} 
		If $\si_i>0$ for $i=1,\dots,m$ then
		$\displaystyle\lim_{x\to \pa(X\setminus \De_X)}|\nabla^g f_g(x)|_g=\infty$.
	\end{lem}
	
	\begin{altproof}{Theorem~\ref{thm:main2}}
		Lemma~\ref{lem:H_INFTY} implies that the set of critical points of $f_g$ is compact. It follows that if $f_g$ is a Morse function, then it has only finitely many critical points, and any $\cC^2$-perturbation of $f_g$ is also a Morse function. Therefore $\mathrm{Riem}_{Morse}^{k+2,\al}(\Si)$ is an open subset of $\mathrm{Riem}^{k+2,\al}(\Si)$. By Theorem~\ref{thm:main1} it is also a dense subset.
	\end{altproof}
	
	\begin{altproof}{Lemma~\ref{lem:H_INFTY}}
		Consider a sequence $x^n \in X\setminus \De_X$ such that
		\[
		x^n \to x^\infty \in \pa(X\setminus \De_X) = \De_X\cup\big((\pa\Si)^l\times(\pa\Si)^{m-l}\big) \subset \Si^l\times(\pa\Si)^{m-l}.
		\]
		{\sc Case 1}. $x^\infty\in\De_X\subset  \intsigma^l\times(\partial\Si)^{m-l}$.\\
		Then, there exists $\xi\in\Si$ such that the set $I:=\big\{i\in\{1,\dots,m\}: x^\infty_i=\xi\big\}$ contains at least two elements. Moreover, if $\xi\in \intsigma$, then $I\subset\{1,\cdots,l\}$; if $\xi\in \partial\Si$, then $I\subset \{l+1,\cdots, m\}$. In either case, we have that
		\[
		\big|\nabla^g_{x_i}R^g(x^n_i)\big| = \cO(1)\text{ and } \big|\nabla^g_{x_i}G^g(x^n_i,x^n_j)\big| = \cO(1)\quad \text{ as } n\to+\infty,
		\] 
		for $i\in I$ and $j\notin I$, and $\big|\nabla^g_{x_i}H^g(x^n_i,x^n_j)\big| = \cO(1)$ for $i,j \in I$. Let $y_\xi:U(\xi)\to B^\xi\subset\R^2$ be the isothermal chart  with $y_{\xi}(\xi)=0$ introduced in Section~\ref{sec:prelim}.   Then for $i\in I$, setting $y^n_j := y_\xi(x^n_j) \in \R^2$ for $j\in I$ and assuming $y^n_i \ne 0$, we obtain as $n\to\infty$:
		\[
		\begin{aligned}
			\big|\nabla^g f_g(x^n)\big|_g
			&\ge \big|\nabla^g_{x_i} f_g(x^n)\big|_g
			= 2\left|\nabla^g_{x_i}\sum_{j\in I\setminus\{i\}}\si_i\si_jG^g(x^n_i,x^n_j)\right|_g + \cO(1)\\
			&= \frac2{\ka(\xi)}\left|\nabla^g_{x_i}\sum_{j\in I\setminus\{i\}}\si_i\si_j\log\big|y_\xi(x^n_i)-y_\xi(x^n_j)\big|\right|_g + \cO(1)\\
			&\ge \frac{2c}{\ka(\xi)}\left|\nabla_{y_i}\sum_{j\in I\setminus\{i\}}\si_i\si_j\log\big|y^n_i-y^n_j\big|\right| + \cO(1)\\
			&= \frac{2c}{\ka(\xi)}\left|\sum_{j\in I\setminus\{i\}}\si_i\si_j\frac{y^n_i-y^n_j}{|y^n_i-y^n_j|^2}\right| + \cO(1)\\
			&\ge \frac{2c}{\ka(\xi)}\left|\sum_{j\in I\setminus\{i\}}\si_i\si_j\left\langle\frac{y^n_i-y^n_j}{|y^n_i-y^n_j|^2},\frac{y^n_i}{|y^n_i|}\right\rangle\right| + \cO(1)\
		\end{aligned}
		\]
		The second inequality is a consequence of the fact that there exists  $c>0$ such that for any function $\mathrm{g}: U(\xi)\to \R$:
		\[
		\big|\nabla^g_x \mathrm{g}(x)\big|_g 
		\ge c\big|\nabla \big(\mathrm{g}\circ y^{-1}_\xi\big)\big(y_{\xi}(x)\big)\big|
		\quad\text{for } x\in U(\xi).
		\]
		Now for $n\in\N$ we choose $i(n)\in I$ with $|y^n_{i(n)}| \ge |y^n_j|$ for all $j\in I$. This implies $y^n_{i(n)} \ne 0$ and
		\[
		\big\langle y^n_{i(n)}-y^n_j,y^n_{i(n)}\big\rangle \ge \frac12\big|y^n_{i(n)}-y^n_j\big|^2 > 0\quad\text{for $j\in I\setminus\{i(n)\}$,}
		\]
		hence
		\[
		\left\langle\frac{y^n_{i(n)}-y^n_j}{|y^n_{i(n)}-y^n_j|^2},\frac{y^n_{i(n)}}{|y^n_{i(n)}|}\right\rangle \ge \frac1{2|y^n_{i(n)}|}
		\quad\text{for $j\in I\setminus\{i(n)\}$.}
		\]
As a consequence we obtain, using $\si_i\si_j>0$ for all $i,j\in I$:
	\[
		\big|\nabla^g f_g(x^n)\big|_g
		\ge \big|\nabla^g_{x_{i(n)}} f_g(x^n)\big|_g 
		\ge \frac{c}{\ka(\xi)}\pi\sum_{j\in I\setminus\{i(n)\}}\si_{i(n)}\si_j\frac1{|y^n_{i(n)}|} + \cO(1) \to \infty \quad\text{as $n\to\infty$.}
	\]
		{\sc Case 2. } $x^{\infty}\notin \De_X$. Then there exists  $i\in \{1,...,l\}$ such that $x^n_i\rightarrow \xi$ for some $\xi\in \partial \Si$. \\
		We fix an isothermal chart $(y_{\xi}, U({\xi}))$ around $\xi$ as introduced in Section~\ref{sec:prelim}. For any $\zeta\in U_{r_{\xi}}({\xi}) $, we decompose  $G^g(\cdot,\zeta)$ as follows: 
		\begin{equation}\label{eq:decompose}
			G^g(\cdot,\zeta)= \tilde{H}^g(\cdot,\zeta)-\frac 1 {\ka(\zeta)}\chi(4|y_{\xi}(\cdot))-y_{\xi}(\zeta)|/r_{\xi})\log {|y_{\xi}(\cdot)-y_{\xi}(\zeta) |}.
		\end{equation}
		Equation \eqref{eq:robin} implies $R^g(\zeta)=\tilde{H}^g(\zeta,\zeta)+ \frac1{2\ka(\zeta)}\vphi_\xi\big(y_\xi(\zeta)\big)$. Let $\pa_1,\pa_2$ be the standard basis of $\R^2$ and define $\pa_{\zeta_1},\pa_{\zeta_2} \in T_{\zeta}\Si$ be the corresponding basis of the tangent space of $\Si$ at $\zeta\in U(\xi)$.
		
		Fix $\zeta\in U_{r_{\xi}}(\xi)\cap \intsigma $ and set $\delta:=y_{\xi}(\zeta)_2>0$. The representation formula for $\tilde{H}^g$ yields:
		\begin{equation}\label{eq:vanish_direction1}
			\begin{aligned}
				&\pa_{\zeta_1} \tilde{H}^g(\eta,\zeta) |_{\eta=\zeta}
				= \frac 1 {|\Si|_g} \int_{\Si}  \pa_{\zeta_1} \tilde{H}^g(\cdot,\zeta)dv_g
				- \int_{\Si}G^g(\cdot,\zeta) \De_g   \pa_{\zeta_1}\tilde{H}^g(\cdot,\zeta) dv_g\\
				&\hspace{4cm}  + \int_{\pa\Si} G^g(\cdot,\zeta)\pa_{\nu_g}   \pa_{\zeta_1} \tilde{H}^g(\cdot,\zeta) ds_g\\
				&\hspace{.5cm}= \frac 1 {2\pi} \int_{\pa\Si} 
				G^g(\zeta,x) \pa_{\nu_g}\pa_{\zeta_1}\left(\chi\left( {4|y_{\xi}(x)-y_{\xi}(\zeta)|} /r_{\xi}\right)\log {|y_{\xi}(x)-y_{\xi}(\zeta) |}\right)ds_g(x)
				+ \cO(1) \\
				&\hspace{.5cm}= - \frac 1 {2\pi}\int_{B^{\xi}\cap \partial\R^2_{+}} 
				G^g\left(\zeta, y_{\xi}^{-1}(y)\right) \chi(4|y-y_{\xi}(\zeta)|/r_{\xi}) \frac{ 2(y_1-y_{\xi}(\zeta)_1)(y_2-y_{\xi}(\zeta)_2)}{ |y-y_{\xi}(\zeta)|^4} dy         
				+ \cO(1)\\
				&\hspace{.5cm}= -\frac 1 {2\pi\delta }\int^{\eps/\delta}_{-\eps/\delta}
				G^g\left(\zeta, y_{\xi}^{-1}\big((\delta  s,0)+ (y_{\xi}(\zeta)_1, 0)\big)\right) \frac{2s}{ (s^2+1)^2} ds+\cO\left(1\right).
			\end{aligned}
		\end{equation} 
		as $\de\to0$. Here $\eps\in (0,\frac{r_{\xi}}{16})$ is chosen sufficiently small. Decomposing $G^g$ as in \eqref{eq:decompose}, we deduce for $\zeta\in U(\xi)$ with $y_{\xi}(\zeta)_2< \frac{r_{\xi}}{16}$:
		\begin{equation*}
			G^g\left(\zeta, y_{\xi}^{-1}\big((\delta  s,0)+ (y_{\xi}(\zeta)_1, 0)\big)\right)
			= 	\tilde{H}^g\left(\zeta, y_{\xi}^{-1}\big((\delta  s,0)+ (y_{\xi}(\zeta)_1, 0)\big)\right) -\frac 1 \pi \log |(\delta s, -\delta)|,
		\end{equation*}
		Now we apply the mean value theorem for $\tilde{H}^g$ and obtain as $\de s\to0$:
		\begin{equation}\label{eq:H^g_expansion}
			\tilde{H}^g\big(\zeta, y_{\xi}^{-1}\big((\delta  s,0)+ (y_{\xi}(\zeta)_1, 0)\big)\big)
			= \tilde{H}^g\left(\zeta,y_{\xi}^{-1}\big(y_{\xi}(\zeta)_1,0\big)\right)
			+ \cO\big(|\delta s|\sup_{x\in \partial\Si} \|\nabla_x\tilde{H}^g(\cdot,x)\|_{\cC(\Si)}\big) 
		\end{equation}
		This implies as $\de\to0$:
		\begin{equation*}
			\begin{aligned}
				&\left|-\frac 1 {2\pi\de}\int^{\eps/\delta}_{-\eps/\de} G^g(\zeta, y_{\xi}^{-1}\big((\de  s,0)+ (y_{\xi}(\zeta)_1, 0)\big) \frac{2s}{ (s^2+1)^2} ds\right|\\
				&\hspace{1cm}\le \left|-\frac 1 {2\pi\de}\tilde{H}^g\big(\zeta,y_{\xi}^{-1}\big(y_{\xi}(\zeta)_1,0\big)\big)
				\int_{|s|\le\eps/\de}\frac{2s}{(s^2+1)^2} ds\right| \\
				&\hspace{2cm}+ \left|\cO\left(\int_{|s|\le \eps/\de} \frac{2s^2}{(s^2+1)^2} ds\right)
				+ \frac {1} {2\pi^2\de}\int_{|s|\le \eps/\de} \log(\de\sqrt{s^2+1}) \frac{2s}{(s^2+1)^2} ds\right|\\
				&\hspace{1cm}\leq  \cO(1).
			\end{aligned}
		\end{equation*}
		Now \eqref{eq:vanish_direction1} yields $\pa_{\zeta_1}\tilde{H}^g(\zeta,\zeta)=\cO(1)$ for $\zeta\in U_{r_{\xi}}(\xi)\cap \pa\Si$. 	Consequently, for $\zeta\in U_{r_{\xi}}(\xi)$ with $y_{\xi}(\zeta)_2< \frac{r_{\xi}}{16}$, we have proven that:
		\begin{equation}\label{eq:H^g_direction_1}
			\pa_{\zeta_1}R^g(\zeta)= \cO(1)\qquad\text{as $|y_{\xi}(\zeta)|\rightarrow 0$.} 
		\end{equation}
		The representation formula of $\tilde{H}^g$ yields for  $\zeta\in  U_{r_{\xi}}(\xi)\cap\intsigma$ as $\de\to0$:
		\begin{equation*}
			\begin{aligned}
				&\pa_{\zeta_2} \tilde{H}^g(\eta,\zeta) |_{\eta=\zeta}
				=\frac 1 {|\Si|_g} \int_{\Si}  \pa_{\zeta_2} \tilde{H}^g(\cdot,\zeta)dv_g
				- \int_{\Si}G^g(\cdot,\zeta) \De_g   \pa_{\zeta_2}\tilde{H}^g(\cdot,\zeta) dv_g\\
				&\hspace{3cm}+ \int_{\pa\Si} G^g(\cdot,\zeta)\pa_{\nu_g}   \pa_{\zeta_2} \tilde{H}^g(\cdot,\zeta)  ds_g\\
				&\hspace{.5cm}= \frac 1 {2\pi} \int_{\pa\Si}  G^g(\zeta,x) \pa_{\nu_g}  \pa_{\zeta_2}\left(\chi\left( {4|y_{\xi}(x)-y_{\xi}(\zeta)|} /r_{\xi}\right)
				\log {|y_{\xi}(x)-y_{\xi}(\zeta) |}\right) ds_g(x)+ \cO(1) \\
				&\hspace{.5cm} \le \frac1{2\pi}\int_{B^{\xi}\cap \pa\R^2_{+}} G^g(\zeta, y_{\xi}^{-1}(y)) \chi(4|y-y_{\xi}(\zeta)|/r_{\xi})
				\frac{(y_1- y_{\xi}(\zeta)_1)^2- (y_2- y_{\xi}(\zeta)_2)^2}{ |y-y_{\xi}(\zeta)|^4} dy+\cO(1)  \\
				&\hspace{.5cm}\le \frac1{2\pi\de}\int^{\eps/\de}_{-\eps/\de}G^g(\zeta, y_{\xi}^{-1}((\de  s,0)+(y_{\xi}(\zeta)_1, 0)) \frac{s^2-1}{ (s^2+1)^2}ds
				+\cO(1) \\
				&\hspace{.5cm} \stackrel{\eqref{eq:H^g_expansion}}{\le}
				\frac1{2\pi\de} \tilde{H}^g(\zeta,y_{\xi}^{-1}(y_{\xi}(\zeta)_1,0))\int_{|s|\leq\eps/\de} \frac{s^2-1}{(s^2+1)^2} ds
				+ \cO\left(1+\int_{|s|\leq \varepsilon/\delta} \frac{|s(s^2-1)|}{(s^2+1)^2} ds\right) \\
				&\hspace{1.5cm}+ \frac {\log (\de^{-1})} {2\pi^2\de}\int_{|s|\le\eps/\de} \frac{s^2-1}{(s^2+1)^2} ds
				+ \frac1{2\pi^2\de}\int_{|s|\le \eps/\de} \log\left(\frac1{\sqrt{s^2+1}}\right)\frac{s^2-1}{(s^2+1)^2} ds +\cO(1)\\
				&\hspace{.5cm}\le \frac1{2\pi^2\de}\int_{|s|\geq 1} \log\left(\frac {\sqrt{s^{-2}+1}} {\sqrt{s^2+1}}\right)\frac{s^2-1}{(s^2+1)^2} ds
				- \frac1{2\pi^2\de}\int_{|s|\ge \eps/\de} \log\left(\frac1{\sqrt{s^2+1}}\right)\frac{s^2-1}{(s^2+1)^2} ds\\
				&\hspace{1.5cm} - \frac {\log(\de^{-1})\eps} {\pi^2(\de^2+\eps^2)}+\cO(\de^{-1/2}) +\cO(\log(\de^{-1})+1) \\
				&\hspace{.5cm}\le - \frac1{2\pi^2\de}\int_{|s|\ge1} \log\left(s^{2}\right)\frac{s^2-1}{(s^2+1)^2} ds + \cO(\de^{-\frac1 2})\le -\frac 1 {4\pi\de}.
			\end{aligned}
		\end{equation*} 
		The last inequality used the identity 
		\[
		\int_{|s|\geq 1} \log(s^2) \frac{s^2-1}{(s^2+1)^2} ds =\pi.
		\]
		From the above estimate we deduce for $\zeta \in \intsigma\cap U_{r_{\xi}}(\xi)$:
		\begin{equation}\label{eq:nagative_R_g}
			\pa_{\zeta_2} R^g(\zeta)\le -\frac1{2\pi |y_{\xi}(\zeta)_2|}\qquad\text{as $|y_{\xi}(\zeta)_2|\to 0$.} 
		\end{equation}
		Now if $I:=\big\{i\in\{1,\dots,m\}: x^n_i\to \xi\big\}$ contains only a single element $i$ then \eqref{eq:nagative_R_g} yields
		\[
		\begin{aligned}
			\big|\nabla^g f_g(x^n)\big|_g
			&\ge \big|\nabla^g_{x_{i}} f_g(x^n)\big|_g
			\ge \big|\si_i^2\nabla^g_{x_{i}} R^g(x_i^n) + 2\si_i\sum_{j\ne i} \nabla^g_{x_{i}} G^g(x^n_i,x^n_j)+ \nabla^g_{x_{i}}h(x^n)\big|_g\\
			&\ge \si_i^2 |\nabla^g_{x_i}R^g(x^n_i)|+\cO(1) 
			\ge  \frac{ c\si_i^2 }{2\pi|y_{\xi}(x^n_i)_2|}+\cO(1)\\
			&\to\infty \qquad\text{as $n\to\infty$.}
		\end{aligned}
		\]
		Here $c>0$ is a constant such that for any function $F:U(\xi)\to \R$:
		\begin{equation}\label{eq:nablaF}
			\big|\nabla^g_x F(x)\big|_g \ge c\big|\nabla \big(F\circ y^{-1}_\xi\big)\big(y_{\xi}(x)\big)\big| \quad\text{for } x\in  U(\xi).
		\end{equation}
		
		Next we  consider the case that $I$ contains at least two elements. Then $\xi^n_i\in U(\xi)$ for $n$ large and $i\in I$. By a direct calculation we obtain for $j\in I\setminus\{i\}$, $\iota=1,2$ and $n\to\infty$:
		\begin{equation}\label{eq:partial_G^g}
			\begin{aligned}
				\pa_{\zeta_\iota} G^g(x^n_j,\zeta) |_{\zeta=x^n_{i}}
				&= \pa_{\zeta_\iota} \tilde{H}^g(x^n_j,\zeta) |_{\zeta=x^n_{i}}\\
				&\hspace{1cm}+ \frac1{\ka(x^n_i)} \log \frac 1 {|y_{\xi}(x^n_j)- y_{\xi}(x_{i}^n)|}
				\pa_{\zeta_\iota}\chi(4|y_{\xi}(x^n_j)- y_{\xi}(\zeta)|/r_{\xi})|_{\zeta=x_{i}^n} \\
				&\hspace{1cm}+ \left.\frac 1 {\ka(x^n_i)} \chi(4|y_{\xi}(x^n_j)- y_{\xi}(x_{i}^n)|/r_{\xi})\pa_{\zeta_\iota}
				\log\frac 1 {|y_{\xi}(x^n_j)- y_{\xi}(\zeta)|}\right|_{\zeta=x_{i}^n}  \\
				&= \pa_{\zeta_\iota} \tilde{H}^g(x^n_j,\zeta) |_{\zeta=x^n_{i}}
				-\frac 1{\ka(x^n_i)}\frac {(y_{\xi}(x_{i}^n)- y_{\xi}(x^n_j) )_{\iota}}{|y_{\xi}(x^n_j) - y_{\xi}(x_{i}^n)|^2} +\cO(1).
			\end{aligned}
		\end{equation}
		For $n\in\N$ we set  
		\[
		\varrho_n:=\max \Big(\big\{  y_{\xi}(x^{n}_i)_2:  i\in I\big\}\cup\big\{ |y_{\xi}(x_{i}^n)_1-y_{\xi}(x^{n}_j)_1|: i,j\in I\text{ with } i\neq j \big\}\Big).
		\]
		If there exists $i(n)\in I$ such that $\varrho_n=y_{\xi}\big(x^n_{i(n)}\big)_2$, then $i(n)\in \{1,\dots,l\}$ satisfies
		\[
		y_{\xi}(x^n_{i(n)})_2-y_{\xi}(x^n_{j})_2\ge 0 \text{ and } y_{\xi}(x^n_{i(n)})_2\ge |(y_{\xi}(x_{i}^n)_1-y_{\xi}(x^{n}_j)_1|
		\quad\text{for every $i,j\in I$.}
		\]
		Given $\zeta=x^n_{i(n)}$ and $\eta= x^n_j$ with $j\in I\cap\{1,\dots,l\}\setminus\{i(n)\})$, we will calculate the upper bound of 
		$\pa_{\zeta_2}\tilde{H}^g(\eta,\zeta)$ as $n\to0$. Setting $a=y_{\xi}(\eta)_2>0$ and $b=y_{\xi}(\zeta)_1-y_{\xi}(\eta)_1$ we have for $|b|\le \varrho_n$ as $n\to\infty$:
		\[
		\begin{aligned}
			&\pa_{\zeta_2} \tilde{H}^g(\eta,\zeta)
			=\frac 1 {|\Si|_g} \int_{\Si}  \pa_{\zeta_2} \tilde{H}^g(\cdot,\zeta)dv_g
			-\int_{\Si}G^g(\cdot,\eta) \De_g \pa_{\zeta_2}\tilde{H}^g(\cdot,\zeta) dv_g\\
			&\hspace{3cm}+ \int_{\pa\Si} G^g(\cdot,\eta)\pa_{\nu_g} \pa_{\zeta_2} \tilde{H}^g(\cdot,\zeta)  ds_g\\
			&\hspace{.5cm}= \cO(1) +\frac1{2\pi} \int_{\pa\Si} G^g(\eta,x) \pa_{\nu_g}\pa_{\zeta_2}
			\big(\chi\big({4|y_{\xi}(x)-y_{\xi}(\zeta)|} /r_{\xi}\big)\log {|y_{\xi}(x)-y_{\xi}(\zeta) |}\big)  ds_g(x)\\
			&\hspace{.5cm}\le \cO(1) + \frac1{2\pi}\int_{B^{\xi}\cap \partial\R^2_{+}} G^g(\eta, y_{\xi}^{-1}(y)) \chi(4|y-y_{\xi}(\zeta)|/r_{\xi})
			\frac{(y_1- y_{\xi}(\zeta)_1)^2-(y_2- y_{\xi}(\zeta)_2)^2}{ |y-y_{\xi}(\zeta)|^4} dy\\
			&\hspace{.5cm}\le \cO(1) +\frac1{2\pi\varrho_n }\int^{\eps/\varrho_n}_{-\eps/\varrho_n}
			G^g(\eta, y_{\xi}^{-1}((\varrho_n  s,0)+ (y_{\xi}(\zeta)_1, 0)) \frac{ s^2-1}{ (s^2+1)^2} ds \\
			&\hspace{.5cm}\stackrel{\eqref{eq:H^g_expansion}}{\le} \frac1{2\pi\varrho_n} \tilde{H}^g(\eta,y_{\xi}^{-1}(y_{\xi}(\zeta)_1,0))
			\int_{|s|\le\eps/\varrho_n} \frac{s^2-1}{(s^2+1)^2} ds
			+ \cO\left(1+\int_{|s|\le \eps/\varrho_n} \frac{|s(s^2-1)|}{(s^2+1)^2} ds\right)\\
			&\hspace{1.5cm}  -\frac {\log(a)} {2\pi^2\varrho_n}\int_{|s|\leq \eps/\varrho_n} \frac{s^2-1}{(s^2+1)^2} ds
			+ \frac1{2\pi^2\varrho_n}
			\int_{|s|\le\eps/\varrho_n} \log\left(\frac1{\sqrt{1+(\varrho_n s+b )^2/a^2}}\right)\frac{s^2-1}{(s^2+1)^2} ds\\
			&\hspace{.5cm}\leq -\frac1{2\pi^2\varrho_n}\int_{|s|\ge \eps/\varrho_n}
			\left(\log (a^{-1})+ \log\left(\frac1{\sqrt{1+(\varrho_n s+b )^2/a^2}}\right)\right)\frac{s^2-1}{(s^2+1)^2} ds\\
			&\hspace{1.5cm} +\frac1{2\pi^2\varrho_n}\int_{|s|\ge1}
			\log\left(\frac {\sqrt{1+(\varrho_n s^{-1}+b )^2/a^2}} {\sqrt{1+(\varrho_n s+b )^2/a^2}}\right)\frac{s^2-1}{(s^2+1)^2} ds
			+\cO\big(1+\log(\varrho_n^{-1})\big)\\
			&\hspace{.5cm}\le \cO\big(\log(\varrho_n^{-1})\big) 
		\end{aligned}
		\]
		Here we used the inequalities:  
		\[
		\log\left(\frac {\sqrt{1+(\varrho_n s^{-1}+b )^2/a^2}} {\sqrt{1+(\varrho_n s+b )^2/a^2}}\right)\leq  0, \text{ for $|b|\leq \varrho_n$, $|s|\ge1$}
		\]
		and 
		\[
		0\le \log \sqrt{a^2+(\varrho_n s+b)^2}\frac{s^2-1}{(s^2+1)^2} \leq C \frac{\varrho_n ^{\frac 12 }}{s^{\frac 3 2}}
		\]
		for some constant $C>0$, any $s\in\{ s\in\R:|s|\ge\eps/\varrho_n,\ a^2+(\varrho_ns+b)^2\ge1\}$. We notice that we have for $j\in I\cap\{l+1,\dots,m\}$:
		\begin{equation}\label{eq:H^G_b_2}
			|\pa_{\zeta_2} \tilde{H}^g(x^n_j,\zeta) |_{\zeta=x^n_{i(n)}}|= |\pa_{\zeta_2} \tilde{H}^g(\zeta,x^n_j) |_{\zeta=x^n_{i(n)}}|\le \cO(1)
			\qquad\text{as $n\to\infty$;}
		\end{equation}  
		and for $j\in I\cap\{1,\dots,l \}\setminus\{ i(n)\}$, $\varrho_n=y_{\xi}(x^n_{i(n)})$:
		\begin{equation}\label{eq:H^G_int_2}
			\big|\pa_{\zeta_2} \tilde{H}^g(x^n_j,\zeta) |_{\zeta=x^n_{i(n)}}\big|\le\cO\big(\log(\varrho_n^{-1})\big) \qquad\text{as $n\to\infty$.}
		\end{equation} 
		Now \eqref{eq:nagative_R_g}-\eqref{eq:H^G_int_2} imply for $n\to\infty$:
		\[
		\begin{aligned}
			|\nabla^g f_g(x^n)|_g
			&\ge \left|\nabla^g_{x_{i(n)}}\left( \si_i^2 R^g(x^n_{i(n)})+2 \si_{i(n)} \sum_{j\ne i(n)} \si_jG^g(x^n_{i(n)}, x^n_j)+ h(x^n)\right)\right|_g\\
			&\ge c \left|\pa_{(x_{i(n)})_2} \left(\si^2_{i(n)}R^g(\xi^n_{i(n)})+ 2\si_{i(n)}\sum_{j\in I\setminus\{i(n)\} }\si_j G^g(\xi_{i(n)}, \xi_j)) \right)\right|
			+\cO(1)\\
			&\ge \cO\big(\log(\varrho_n^{-1})\big) + \si^2_{i(n)} \frac{ c }{2\pi\varrho_n}\\
			&\to \infty.
		\end{aligned}
		\]
		Here $c>0$ is as in \eqref{eq:nablaF}. If $\varrho_n>\max\{y_{\xi}(x^n_i): i\in I\}$, then we  take $i(n)\in I$ such that $y_{\xi}(x^n_{i(n)})_1=\max\{ y_{\xi}(x^n_i): i\in I\}$. The proof of Lemma 6 in~\cite{yang_zhou2021} implies for $x\in U(\xi)$ and $y^*_{\xi}(x):= (y_{\xi}(x)_1, -y_{\xi}(x)_2)$:
		\[
		\left\| G^g(\cdot,x)+\frac1{2\pi}\log\big|y_{\xi}(\cdot)-y_{\xi}(x)\big|
		+\frac 1 {2\pi}\log\big|y_{\xi}(\cdot)-y^{*}_{\xi}(x)\big|\right\|_{\cC^1(U_{r_{\xi}}(\xi))}\le C
		\]
		where $C>0$ depends only on $(\Si,g)$ and on $\xi$. Therefore we have for $j\in I\setminus\{i(n)\}$ as $n\to\infty$:
		\begin{equation}\label{eq:partial_G^g_direct1}
			\pa_{\zeta_1} G(x^n_j, \zeta)\big|_{\zeta=x^n_{i(n)}}
			\le -\frac1{2\pi} \frac{y_{\xi}(x^n_{i(n)})_1- y_{\xi}(x^n_j)_1}{|y_{\xi}(x^n_{i(n)})- y_{\xi}(x^n_j)|^2}
			-\frac1{2\pi} \frac{y_{\xi}(x^n_{i(n)})_1- y_{\xi}(x^n_j)_1}{|y_{\xi}(x^n_{i(n)})- y^*_{\xi}(x^n_j)|^2}+ \cO(1). 
		\end{equation}
		We can take  $i'(n)\in I\setminus\{ i(n)\}$ such that $\varrho_n= y_{\xi}(x^n_{i(n)})_1-y_{\xi}(x^n_{i'(n)})_1$ by the assumption. The inequality \eqref{eq:partial_G^g_direct1} yields
		\begin{equation}\label{eq:eq:partial_G^g_direct1_negative}
			\pa_{\zeta_1} G(x^n_{i'(n)}, \zeta)\big|_{\zeta=x^n_{i(n)}}\le -\frac1{4\pi\varrho_n}\qquad\text{as $n\to\infty$.}
		\end{equation}
		From \eqref{eq:H^g_direction_1} together with \eqref{eq:partial_G^g_direct1} and \eqref{eq:eq:partial_G^g_direct1_negative} we derive the following estimate for the gradient of $f_g$ as $n\to\infty$:
		\[
		\begin{aligned}
			|\nabla^g f_g(x^n)|_g
			&\ge \left|\nabla^g_{x_{i(n)}}\left( \si_i^2 R^g(x^n_{i(n)})+2 \si_{i(n)} \sum_{j\neq i(n)} \si_jG^g(x^n_{i(n)}, x^n_j)+ h(x^n)\right)\right|_g\\
			&\ge c \left|\pa_{(x_{i(n)})_1}
			\left(\sigma^2_{i(n)}R^g(\xi^n_{i(n)})+ 2\sigma_{i(n)}\sum_{j\in I\setminus\{i(n)\} }\sigma_j G^g(\xi_{i(n)}, \xi_j)) \right)\right|+\cO(1)\\
			&\ge \cO(1)+ \si_{i(n)}\si\big(i'(n)\big) \frac{ c }{2 \pi\varrho_n}\\
			&\to \infty 
		\end{aligned}
		\]
		Again $c>0$ is as in \eqref{eq:nablaF}. This completes the proof of Lemma \ref{lem:H_INFTY}.
	\end{altproof}	
	
	\bibliographystyle{acm} 
	\bibliography{Ba-Hu_Morse_final} 

\begin{thebibliography}{10}

\bibitem{Ahmedou-Bartsch-Fiernkranz:2023}
{\sc Ahmedou, M., Bartsch, T., and Fiernkranz, T.}
\newblock Equilibria of vortex type {H}amiltonians on closed surfaces.
\newblock {\em Topol. Methods Nonlinear Anal. 61}, 1 (2023), 239--256.

\bibitem{baraket1998construction}
{\sc Baraket, S., and Pacard, F.}
\newblock Construction of singular limits for a semilinear elliptic equation in
  dimension 2.
\newblock {\em Calc. Var. Partial Differ. Equ. 6}, 1 (1997), 1--38.

\bibitem{Bartsch2017TheMP}
{\sc Bartsch, T., Micheletti, A.~M., and Pistoia, A.}
\newblock The {M}orse property for functions of {K}irchhoff-{R}outh path type.
\newblock {\em Discrete Contin. Dyn. Syst.-S 12}, 7 (2019), 1867--1877.

\bibitem{chern1955}
{\sc Chern, S.-S.}
\newblock An elementary proof of the existence of isothermal parameters on a
  surface.
\newblock {\em Proc. Amer. Math. Soc. 6\/} (1955), 771--782.

\bibitem{del_pino_singular_2005}
{\sc del Pino, M., Kowalczyk, M., and Musso, M.}
\newblock Singular limits in {Liouville}-type equations.
\newblock {\em Calc. Var. Partial Differ. Equ. 24}, 1 (2005), 47--81.

\bibitem{del_pino_collapsing_2006}
{\sc del Pino, M., and Wei, J.}
\newblock Collapsing steady states of the {Keller}–{Segel} system.
\newblock {\em Nonlinearity 19}, 3 (2006), 661--684.

\bibitem{Esposito2014singular}
{\sc Esposito, P., and Figueroa, P.}
\newblock Singular mean field equations on compact {R}iemann surfaces.
\newblock {\em Nonlinear Anal. 111\/} (2014), 33--65.

\bibitem{Esposito2005}
{\sc Esposito, P., Grossi, M., and Pistoia, A.}
\newblock On the existence of blowing-up solutions for a mean field equation.
\newblock {\em Ann. Inst. H. Poincaré Anal. Non Linéaire 22}, 2 (2005),
  227--257.

\bibitem{figueroa2022bubbling}
{\sc Figueroa, P.}
\newblock Bubbling solutions for mean field equations with variable intensities
  on compact {R}iemann surfaces.
\newblock {\em J. d'Analyse Math., {\rm DOI} 10.1007/s11854-023-0303-2\/}
  (2023).

\bibitem{gladiali2004some}
{\sc Gladiali, F., and Grossi, M.}
\newblock Some results for the {G}elfand's problem.
\newblock {\em Commun. Partial Differ. Equ. 29}, 9-10 (2005), 1335--1364.

\bibitem{Hartman1955}
{\sc Hartman, P., and Wintner, A.}
\newblock On uniform {D}ini conditions in the theory of linear partial
  differential equations of elliptic type.
\newblock {\em Amer. J. Math. 77\/} (1955), 329--354.

\bibitem{henry2005perturbation}
{\sc Henry, D.}
\newblock {\em Perturbation of the Boundary in Boundary-Value Problems of
  Partial Differential Equations}.
\newblock London Mathematical Society Lecture Note Series. Cambridge University
  Press, Cambridge, 2005.

\bibitem{ma_convergence_2001}
{\sc Ma, L., and Wei, J.}
\newblock Convergence for a {L}iouville equation.
\newblock {\em Comment. Math. Helv. 76}, 3 (Sep. 2001), 506--514.

\bibitem{Nardi2014}
{\sc Nardi, G.}
\newblock Schauder estimate for solutions of {P}oisson’s equation with
  {N}eumann boundary condition.
\newblock {\em Enseign. Math. 60}, 3/4 (2014), 421--435.

\bibitem{Vekua1955}
{\sc Vekua, I.~N.}
\newblock The problem of reduction to canonical form of differential forms of
  elliptic type and the generalized {C}auchy-{R}iemann system.
\newblock {\em Dokl. Akad. Nauk SSSR (N.S.) 100\/} (1955), 197--200.

\bibitem{yang_zhou2021}
{\sc Yang, Y., and Zhou, J.}
\newblock Blow-up analysis involving isothermal coordinates on the boundary of
  compact {R}iemann surface.
\newblock {\em J. Math. Anal. Appl. 504}, 2 (2021), 125440.

\end{thebibliography}
	
	\vspace{2mm}\noindent
	{\sc Thomas Bartsch, Zhengni Hu}\\
	Mathematisches Institut\\
	Universit\"at Giessen\\
	Arndtstr.\ 2\\
	35392 Giessen, Germany\\
	Thomas.Bartsch@math.uni-giessen.de\\
	Zhengni.Hu@math.uni-giessen.de\\
	
\end{document}